\documentclass[12pt]{article}
\pdfoutput=1

\usepackage{a4wide}
\usepackage{amsmath}
\usepackage{amssymb}
\usepackage{amsthm}
\usepackage{picins}
\usepackage{graphicx}
\usepackage{wrapfig}
\usepackage{subfigure}
\usepackage{color}
\usepackage[small,bf]{caption}
\usepackage{ifpdf}
\ifpdf
\fi

\newtheorem{theorem}{Theorem}
\newtheorem{proposition}{Proposition}

\newtheorem{definition}{Definition}

\newtheorem{conjecture}{Conjecture}

\theoremstyle{definition}
\newtheorem{example}{Example}

\begin{document}

\begin{center}
{\Huge Discrete tomography reconstructions with small boundary}

\bigskip

{\Large Birgit van Dalen}

\textit{Mathematisch Instituut, Universiteit Leiden, Niels Bohrweg 1, 2333 CA Leiden, The Netherlands \\ dalen@math.leidenuniv.nl}

\today

\end{center}

{\small \textbf{Abstract:} We consider the problem of reconstructing binary images from their horizontal and vertical projections. For any reconstruction we define the length of the boundary of the image. In this paper we assume that the projections are monotone, and we construct an image satisfying these projections that has a relatively small boundary. We also give families of examples for which we show that no smaller boundary is possible.}

\section{Introduction}\label{introduction}

An important problem in discrete tomography is to reconstruct a binary image on a lattice from given projections in lattice directions \cite{boek, boeknieuw}. Each point of a binary image
has a value equal to zero or one. The line sum of a line through the image is the sum of the
values of the points on this line. The projection of the image in a certain direction
consists of all the line sums of the lines through the image in this direction. Any binary image with exactly the same projections as the original image we call a \emph{reconstruction} of the image.

For any set of more than two directions, the problem of reconstructing a binary image from its projections in those directions is NP-complete \cite{gardner}. For exactly two directions, the horizontal and vertical ones, say, it is possible to reconstruct an image in polynomial time. Already in 1957, Ryser described an algorithm to do so \cite{ryser}. He also characterised the set of projections that correspond to a unique binary image.

If there are multiple images corresponding to a given set of line sums, it is interesting to reconstruct an image with a special property. In order to find reconstructions that look rather like a real object, two special properties in particular are often imposed on the reconstructions. The first is \emph{connectivity} of the points with value one in the picture \cite{hv-convex2,hv-convex3,woeginger}. The second is \emph{hv-convexity}: if in each row and each column, the points with value one form one connected block, the image is called \emph{hv-convex}. The reconstruction of hv-convex images, either connected or not necessarily connected, has been studied extensively \cite{hv-convex1, hv-convex2, hv-convex3, dahlflatberg, woeginger}.

Another relevant concept in this context is the \emph{boundary} of a binary image. The boundary can be defined as the set of pairs consisting of two adjacent points, one with value 0 and one with value 1. Here we use 4-adjacency: that is, a point is adjacent to its two vertical and to its two horizontal neighbours \cite{connectivity}. The number of such pairs of adjacent points with two different values is called the \emph{length of the boundary} or sometimes the \emph{perimeter length} \cite{gray}.

In this paper we will consider given line sums that may correspond to more than one binary image. Since the boundary of real objects is often small compared to the area, it makes sense to look for reconstructions of which the length of the boundary is as small as possible. In particular, if there exists an hv-convex reconstruction, then the length of the boundary of that image is the smallest possible. In that sense, the length of the boundary is a more general concept than hv-convexity.

In \cite{birgit5} we proved a lower bound on the length of the boundary for any reconstruction of a picture with given line sums. In this paper we complement this result by giving a reconstruction that has a relatively small boundary in the case that both the row and the column sums are monotone.

After introducing some notation in Section \ref{notation}, we describe the construction of a solution to the discrete tomography problem in Section \ref{construction}. In Section \ref{boundary} we prove upper bounds on the length of the boundary of this constructed solution. We show by examples that these bounds are sharp in Section \ref{examples}, and finally in Section \ref{generalisation} we generalise the results slightly.

\section{Definitions and notation}\label{notation}

Let $F$ be a finite subset of $\mathbb{Z}^2$ with characteristic function $\chi$. (That is, $\chi(k,l) = 1$ if $(k,l) \in F$ and $\chi(k,l) = 0$ otherwise.) For $i \in \mathbb{Z}$, we define \emph{row} $i$ as the set $\{(k,l) \in \mathbb{Z}^2: k = i\}$. We call $i$ the index of the row. For $j \in \mathbb{Z}$, we define \emph{column} $j$ as the set $\{(k,l) \in \mathbb{Z}^2: l = j\}$. We call $j$ the index of the column. Note that we follow matrix notation: we indicate a point $(i,j)$ by first its row index $i$ and then its column index $j$. Also, we use row numbers that increase when going downwards and column numbers that increase when going to the right.

The \emph{row sum} $r_i$ is the number of elements of $F$ in row $i$, that is $r_i = \sum_{j \in \mathbb{Z}} \chi(i,j)$. The \emph{column sum} $c_j$ of $F$ is the number of elements of $F$ in column $j$, that is $c_j = \sum_{i \in \mathbb{Z}} \chi(i,j)$. We refer to both row and column sums as the \emph{line sums} of $F$. We will usually only consider finite sequences $\mathcal{R} = (r_1, r_2, \ldots, r_m)$ and $\mathcal{C} = (c_1, c_2, \ldots, c_n)$ of row and column sums that contain all the nonzero line sums. In this paper we will always assume that the line sums are monotone, that is $r_1 \geq r_2 \geq \ldots \geq r_m$ and $c_1 \geq c_2 \geq \ldots \geq c_n$.

Given sequences of integers $\mathcal{R} = (r_1, r_2, \ldots, r_m)$ and $\mathcal{C} = (c_1, c_2, \ldots, c_n)$, we say that $(\mathcal{R}, \mathcal{C})$ is \emph{consistent} if there exists a set $F$ with row sums $\mathcal{R}$ and column sums $\mathcal{C}$. Define $b_i = \#\{j: c_j \geq i\}$ for $i = 1, 2, \ldots, m$. Ryser's theorem \cite{ryser} states that if $r_1 \geq r_2 \geq \ldots \geq r_m$, the line sums $(\mathcal{R}, \mathcal{C})$ are consistent if and only if for each $k = 1, 2, \ldots, m$ we have $\sum_{i=1}^k b_i \geq \sum_{i=1}^k r_i$. Note that by definition of $b_i$ we have $\sum_{i=1}^m b_i = \sum_{j=1}^n c_j = \sum_{i=1}^m r_i$.

We will now define a \emph{uniquely determined neighbour corresponding to line sums} $(\mathcal{R},\mathcal{C})$. This is a uniquely determined set that is in some sense the closest to any set with those line sums. See also \cite[Section 4]{birgit2}.

\begin{definition}
Let row sums $\mathcal{R} = (r_1, r_2, \ldots, r_m)$ and column sums $\mathcal{C} = (c_1, c_2, \ldots, c_n)$ be given, where $n=r_1 \geq r_2 \geq \ldots \geq r_m$ and $m=c_1 \geq c_2 \geq \ldots \geq c_n$. Let $b_i = \#\{j: c_j \geq i\}$ for $i = 1, 2, \ldots, m$. Then the column sums $c_1$, $c_2$, \ldots, $c_n$ and row sums $b_1$, $b_2$, \ldots, $b_m$ uniquely determine a set $F_1$, which we will call the \emph{uniquely determined neighbour corresponding to line sums} $(\mathcal{R},\mathcal{C})$.
\end{definition}

Suppose line sums $\mathcal{R} = (r_1, r_2, \ldots, r_m)$ and $\mathcal{C} = (c_1, c_2, \ldots, c_n)$ are given, where $r_1 \geq r_2 \geq \ldots \geq r_m$ and $c_1 \geq c_2 \geq \ldots \geq c_n$. Let the uniquely determined neighbour corresponding to $(\mathcal{R}, \mathcal{C})$ have row sums $b_1 \geq b_2 \geq \ldots \geq b_n$. Then we define
\[
\alpha(\mathcal{R},\mathcal{C}) = \frac{1}{2} \sum_{i=1}^{m} |r_i - b_i|.
\]
Note that $\alpha(\mathcal{R},\mathcal{C})$ is an integer, since $2\alpha(\mathcal{R},\mathcal{C})$ is congruent to
\[
\sum_{i=1}^m (r_i + b_i) = \sum_{i=1}^m r_i + \sum_{i=1}^m b_i = 2\sum_{i=1}^m r_i \equiv 0 \mod 2.
\]
If we write $d_i = b_i - r_i$ for all $i$, then because $\sum_{i=1}^m r_i = \sum_{i=1}^m b_i$, we have
\[
\alpha = \sum_{d_i>0} d_i = - \sum_{d_i<0} d_i.
\]

We can view the set $F$ as a picture consisting of cells with zeroes and ones. Rather than $(i,j) \in F$, we might say that $(i,j)$ has value 1 or that there is a one at $(i,j)$. Similarly, for $(i,j) \not\in F$ we sometimes say that $(i,j)$ has value zero or that there is a zero at $(i,j)$.

We define the \emph{boundary} of $F$ as the set consisting of all pairs of points $\big( (i,j), (i',j') \big)$ such that
\begin{itemize}
\item $i=i'$ and $|j-j'| =1$, or $|i-i'| = 1$ and $j=j'$, and
\item $(i,j) \in F$ and $(i',j') \not\in F$.
\end{itemize}
One element of this set we call \emph{one piece of the boundary}. We can partition the boundary into two subsets, one containing the pairs of points with $i=i'$ and the other containing the pairs of points with $j=j'$. The former set we call the \emph{vertical boundary} and the latter set we call the \emph{horizontal boundary}. We define the \emph{length of the (horizontal, vertical) boundary} as the number of elements in the (horizontal, vertical) boundary. For a given set $F$ we denote the length of the horizontal boundary by $L_h(F)$ and the length of the vertical boundary by $L_v(F)$.

\section{The construction}\label{construction}

In this section we will construct a set $F_2$ satisfying given monotone row and column sums that are consistent. First we will describe one step of this construction.

Let row sums $\mathcal{R} = (r_1, r_2, \ldots, r_m)$ and column sums $\mathcal{C} = (c_1, c_2, \ldots, c_n)$ be given, such that $n = r_1 \geq r_2 \geq \ldots \geq r_m$ and $m = c_1 \geq c_2 \geq \ldots \geq c_n$. Assume that those line sums are consistent. For $i = 1, 2, \ldots, m$ define $b_i = \#\{j: c_j \geq i\}$ and $d_i = b_i - r_i$. For convenience we define $r_{m+1} = b_{m+1} = d_{m+1} =0$. We have $n = b_1 \geq b_2 \geq \ldots \geq b_m > b_{m+1}$.

Let $F_1$ be the uniquely determined neighbour corresponding to the line sums $(\mathcal{R}, \mathcal{C})$. Then $F_1$ has row sums $(b_1, b_2, \ldots, b_m)$ and column sums $(c_1, c_2, \ldots, c_n)$. Moreover, in every column $j$ the elements of $F_1$ are exactly in the first $c_j$ rows.

If $r_i = b_i$ for all $i$, then $F_1$ already satisfies the line sums $(\mathcal{R},\mathcal{C})$, and there is nothing to be done. Now assume that not for all $i$ we have $r_i = b_i$. Then there is at least one $i$ with $d_i >0$ and one $i$ with $d_i < 0$. Also, because of the consistency of the line sums the smallest $i$ with $d_i \neq 0$ satisfies $d_i > 0$.

Let $i_1$ be minimal such that $d_{i_1} >0$ and let $i_2$ be minimal such that $d_{i_2} > 0$ and $d_{i_2+1} \leq 0$. Let $R^+ = \{i_1, i_1+1, \ldots, i_2\}$. Similarly, let $i_3$ be minimal such that $d_{i_3} < 0$ and let $i_4$ be minimal such that $d_{i_4}<0$ and $d_{i_4+1} \geq 0$. Such $i_4$ exists, since $d_{m+1} = 0$. Let $R^- = \{i_3, i_3+1, \ldots, i_4\}$. Now $d_i > 0$ for all $i \in R^+$ and $d_i < 0$ for all $i \in R^-$.

If $|R^+| \leq |R^-|$, we execute an \textbf{A-step}, while if $|R^+| > |R^-|$, we execute a \textbf{B-step}. We will now describe these two different steps.

\textbf{A-step.} Let $j$ be maximal such that $c_j \in R^+$. Such a $j$ exists, because as $b_{i_2+1} \leq r_{i_2+1} \leq r_{i_2} < b_{i_2}$, there exists a column with sum $i_2$. Define $s= c_j-i_1+1$; this is the number of rows $i$ with $i_1 \leq i \leq c_j$. We will be moving the ones in the $s$ cells $(i_1, j)$, \ldots, $(c_j, j)$ to other cells. To determine to which cells those ones are moved, consider $i_3, i_3+1, \ldots, i_3+s-1$. Since $i_4-i_3+1 = |R^-| \geq |R^+| \geq s$, we have $i_3+s-1 \leq i_4$, so $\{i_3, i_3+1, \ldots, i_3+s-1\} \subset R^-$. If $r_{i_3+s-1} > r_{i_3+s}$, then let $I = \{i_3, i_3+1, \ldots, i_3+s-1\}$.

Now suppose $r_{i_3+s-1} = r_{i_3+s}$. Let $t_1$ be minimal such that $i_3 \leq t_1 \leq i_3+s-1$ and $r_{t_1} = r_{i_3+s-1}$. Let $t_2$ be such that $t_2 \geq i_3+s$ and $r_{i_3+s-1} = r_{t_2} > r_{t_2+1}$. Since we have $d_{i_4+1} \geq 0$, we have $r_{i_4+1} \leq b_{i_4+1} \leq b_{i_4} < r_{i_4}$, hence $t_2 \leq i_4$. Let $t_3 = t_2+t_1-i_3-s+1$. As $t_2 \geq i_3+s$, we have $t_3 \geq t_1+1$, and as $t_1 \leq i_3+s-1$, we have $t_3 \leq t_2$. Now define $I = \{i_3, i_3+1, \ldots, t_1-1\} \cup \{t_3, t_3+1, \ldots, t_2\}$. We have $|I| = (t_1-i_3) + (-t_1+i_3+s) = s$.

In both cases we have now defined a set $I \subset R^-$ with $|I| = s = c_j - i_1+1$ and satisfying the following property: if $i \in I$ and $i+1 \not\in I$, then $r_i > r_{i+1}$.

Now we move the ones from the rows $i$ with $i_1 \leq i \leq c_j$ to the rows $i \in I$. This column will later be one of the columns of $F_2$. We delete the column and change the line sums accordingly: define for $i=1, 2, \ldots, m$ the new row sums $r_i'$, which is equal to $r_i$ if there was no one in this row in column $j$, and equal to $r_i-1$ if there was a one in this row in column $j$. We have
\[
r_i' = \begin{cases} r_i-1 & \text{for $i < i_1$}, \\ r_i & \text{for $i_1 \leq i \leq c_j$}, \\ r_i-1 & \text{for $i \in I$}, \\ r_i & \text{for $i > c_j$ and $i \not\in I$}. \end{cases}
\]
Also let $b_i'$ be the number of columns not equal to $j$ with column sum at least $i$. We have
\[
b_i' = \begin{cases} b_i - 1 & \text{for $i \leq c_j$}, \\ b_i & \text{for $i> c_j$}. \end{cases}
\]
Note that the set $F_1'$, defined as $F_1$ without column $j$, has row sums $b_1', b_2', \ldots, b_m'$.

We now want to show that the new row sums are non-increasing and that they are consistent, together with the column sums without column $j$, that is, that $\sum_{i=1}^k b_i' \geq \sum_{i=1}^k r_i'$ for $k =1, 2, \ldots, m$.

Suppose for some $i$ we have $r_i' < r_{i+1}'$. Then we must have $r_i' = r_i-1$ and $r_{i+1}' = r_{i+1}$, since $r_i \geq r_{i+1}$. So either $i=i_1-1$ or $i \in I$ and $i+1 \not\in I$. In the latter case we know $r_i > r_{i+1}$, hence $r_i' \geq r_{i+1}'$. If on the other hand $i=i_1-1$, we have $d_i = 0$ and $d_{i+1} > 0$, so $r_i = b_i \geq b_{i+1} > r_{i+1}$, hence $r_i' \geq r_{i+1}'$. We conclude that it can never happen that $r_i' < r_{i+1}'$. So $n-1 = r_1' \geq r_2' \geq \ldots \geq r_m'$.

Now we prove consistency. For $i<i_1$ we have $d_i = 0$, hence
\[
b_i' - r_i' = (b_i-1)-(r_i-1)= d_i = 0.
\]
For $i_1 \leq i \leq c_j$ we have $d_i > 0$, hence
\[
b_i' - r_i' = (b_i-1) - r_i = d_i - 1 \geq 0.
\]
For $c_j+1 \leq i \leq i_3 -1$ we have $d_i \geq 0$, hence
\[
b_i' - r_i' = b_i - r_i = d_i \geq 0.
\]
So for $k \leq i_3 - 1$ we clearly have
\[
\sum_{i=1}^k (b_i' - r_i') \geq 0.
\]
On the other hand, for $k \geq i_4$ we have $\sum_{i=1}^k (b_i-r_i) \geq 0$ because of the consistency of the original line sums, hence
\[
\sum_{i=1}^k (b_i' - r_i') = \left( \sum_{i=1}^k b_i - c_j \right) - \left( \sum_{i=1}^k r_i - c_j \right) = \sum_{i=1}^k (b_i-r_i) \geq 0.
\]
For $i_3 \leq i \leq i_4$ we have $d_i < 0$, so
\[
b_i' - r_i' = b_i-r_i = d_i <0 \qquad \text{if $i \not\in I$},\] \[ b_i' - r_i' = b_i - (r_i -1) = d_i +1 \leq 0 \qquad \text{if $i \in I$}.
\]
Hence for $i_3 \leq k \leq i_4-1$ we have
\[
\sum_{i=1}^k (b_i' - r_i') = \sum_{i=1}^{i_4} (b_i' - r_i') - \sum_{i=k+1}^{i_4} (b_i'-r_i') \geq 0.
\]
This proves the consistency.

\textbf{B-step.} Let $j$ be minimal such that $c_j+1 \in R^-$. Such a $j$ exists, because as $b_{i_3-1} \geq r_{i_3-1} \geq r_{i_3} > b_{i_3}$, there exists a column with sum $i_3-1$. Similarly to the A-step, we find a set $I \subset R^+$ such that $|I| = i_4 - c_j$ with the following property: if $i \not\in I$ and $i+1 \in I$, then $r_i > r_{i+1}$.

Now we move the ones from the rows $i$ with $i \in I$ to the rows $i$ with $c_j+1 \leq i \leq i_4$. This column will later be one of the columns of $F_2$. We delete the column and change the line sums accordingly. Analogously to above we prove that the new line sums are non-increasing and consistent, and that the set $F_1'$ that we have left, is the uniquely determined neighbour corresponding to these new line sums.

The procedure described above, which changes line sums $(\mathcal{R}, \mathcal{C})$ and their uniquely determined neighbour $F_1$ to new line sums $(\mathcal{R'}, \mathcal{C'})$ and their uniquely determined neighbour $F_1'$, we denote by $\varphi$. Since the new line sums satisfy all the necessary properties, we can apply $\varphi$ also to $(\mathcal{R'}, \mathcal{C'})$ and $F_1'$. We can repeat this until we arrive at a situation where the uniquely determined neighbour already satisfies the line sums. One by one we can then put the deleted columns back in the right position (first the column that was last deleted, then the one that was deleted before that, and so on, to make sure that the resulting set $F_2$ has its columns in the right order). Every time we put back a column, the line sums change back to what they were before that instance of $\varphi$ was applied. When all the columns are back in place, the line sums are therefore equal to $(\mathcal{R}, \mathcal{C})$ and the resulting set satisfies these line sums. This proves the following theorem.

\begin{theorem}\label{thmconstruction}
Let row sums $\mathcal{R} = (r_1, r_2, \ldots, r_m)$ and column sums $\mathcal{C} = (c_1, c_2, \ldots, c_n)$ be given, where $n=r_1 \geq r_2 \geq \ldots \geq r_m$ and $m=c_1 \geq c_2 \geq \ldots \geq c_n$. Assume that the line sums are consistent. Let $F_1$ be the uniquely determined neighbour corresponding to the line sums $(\mathcal{R},\mathcal{C})$. If we start with $F_1$ and repeatedly apply $\varphi$ until this is no longer possible, and then put all the deleted columns back in the right position, then the result is a set $F_2$ that satisfies the line sums $(\mathcal{R},\mathcal{C})$.
\end{theorem}

Now we show an example of this construction. Let $m=12$, $n=11$ and define line sums
\[
\mathcal{R} = (11, 10, 8, 8, 8, 6, 6, 6, 3, 3, 3, 2), \qquad \mathcal{C} = (12, 10, 7, 6, 6, 6, 6, 6, 6, 6, 3).
\]
We have
\[
(b_1, \ldots, b_{12}) = (11, 11, 11, 10, 10, 10, 3, 2, 2, 2, 1, 1), \]\[(d_1, \ldots, d_{12}) =  (0, +1, +3, +2, +2, +4, -3, -4, -1, -1, -2, -1).
\]
We will now do the construction step by step, illustrated by Figure \ref{figconstructie}. The $r_i$ and $d_i$ in every step are indicated in the figure. We start with the uniquely determined neighbour $F_1$, that is, the set with column sums $\mathcal{C}$ and row sums $(b_1, \ldots, b_{12})$.

\begin{figure}
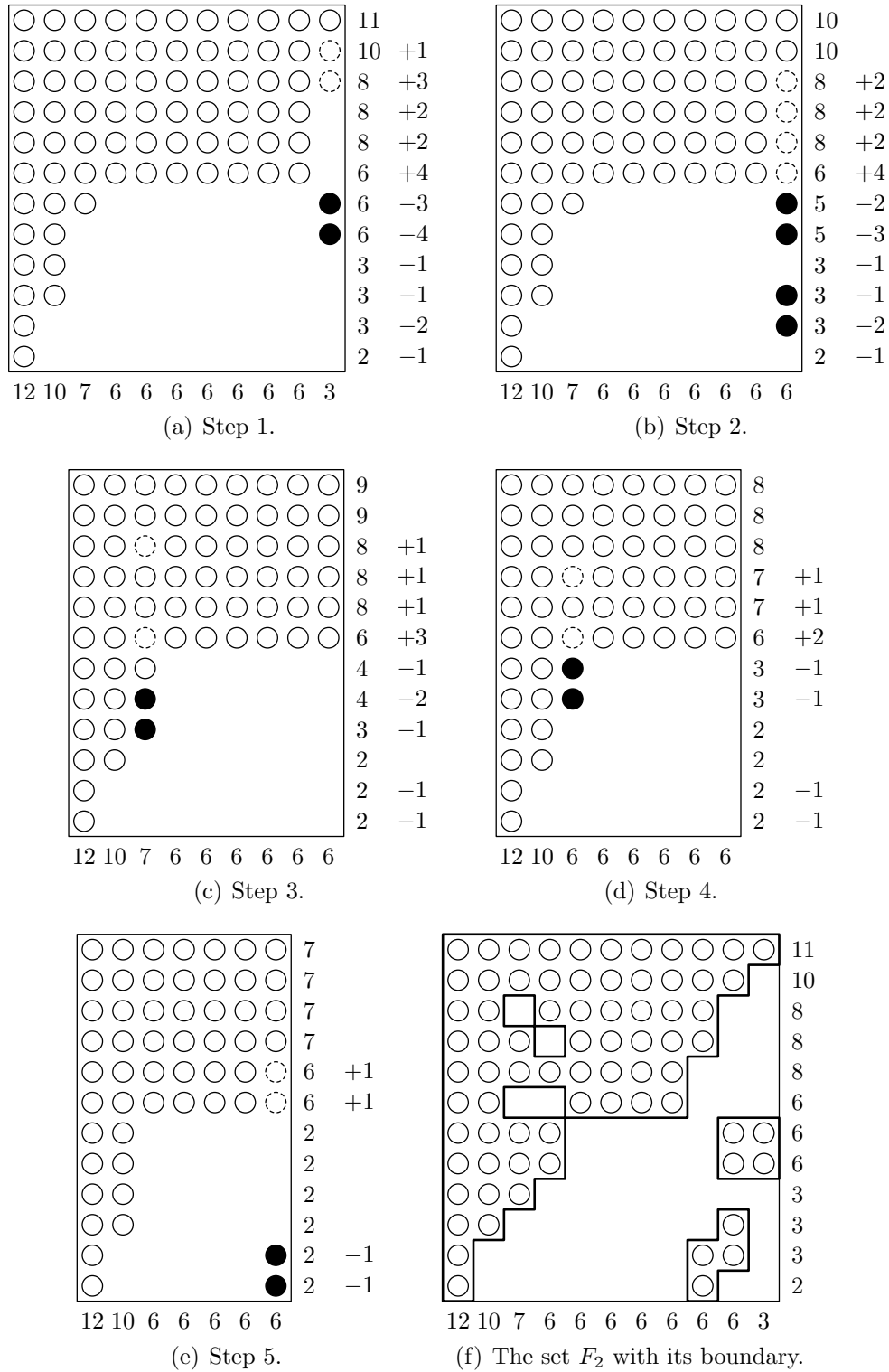

  \begin{center}
    \subfigure[Step 1.]{\includegraphics{plaatje.4}}
    \qquad
    \subfigure[Step 2.]{\includegraphics{plaatje.5}}
    \qquad
    \subfigure[Step 3.]{\includegraphics{plaatje.6}}
    \qquad
    \subfigure[Step 4.]{\includegraphics{plaatje.7}}
    \qquad
    \subfigure[Step 5.]{\includegraphics{plaatje.8}}
    \qquad
    \subfigure[The set $F_2$ with its boundary.]{\label{figconstructieeind}\includegraphics{plaatje.9}}
  \end{center}
  \caption{The construction of the set $F_2$. The ones are indicated by white circles. The dashed circles are ones that are deleted in that step, while the black circles are ones that are newly added in that step. The numbers directly next to each figure are the row sums, while the numbers next to that are the $d_i$.}
  \label{figconstructie}
\end{figure}

\textbf{Step 1.} We have $R^+=\{2, 3, 4, 5, 6\}$, $R^-=\{7, 8, 9, 10, 11, 12\}$. Since $|R^+| \leq |R^-|$, we execute an A-step. The rightmost column $j$ with $c_j \in R^+$ is column 11, with sum 3. We delete the ones in $(2,11)$ and $(3,11)$. We find $I=\{7, 8\}$, since $r_8 > r_9$. So we add ones in $(7,11)$ and $(8,11)$. We then delete column 11.

\textbf{Step 2.} We have $R^+ = \{3, 4, 5, 6\}$ and $R^-=\{7, 8, 9, 10, 11,12\}$. Since $|R^+| \leq |R^-|$, we execute an A-step. The rightmost column $j$ with $c_j \in R^+$ is column 10, with sum 6. We delete the ones in this column in rows 3, 4, 5 and 6. Since $r_{10} = r_{11}$, we cannot use $I = \{7, 8, 9, 10\}$. Instead we take $I= \{7, 8, 10, 11\}$. This works since $r_8 > r_9$ and $r_{11} > r_{12}$. So we add ones in column 10 in rows 7, 8, 10 and 11. We then delete column 10.

\textbf{Step 3.} In row 10, the new row sum is 2, while the new $b_{10}$ is also 2. So the new $d_{10}$ is 0. This means that while $R^+$ is still equal to $\{3, 4, 5, 6\}$, we now have $R^- = \{7, 8, 9\}$. Hence $|R^+| > |R^-|$ and therefore we execute a B-step. The leftmost column $j$ with $c_j +1 \in R^-$ is column 3 with sum 7. So we add ones in $(8,3)$ and $(9,3)$. As $r_5 = r_4 = r_3$, we cannot take $I = \{6, 5\}$, but we have to take $I = \{6, 3\}$. Hence we delete ones in $(3,3)$ and $(6,3)$. We then delete column 3.

\textbf{Step 4.} We have $R^+=\{4, 5, 6\}$ and $R^- = \{7, 8\}$. As $|R^+| > |R^-|$, we execute a B-step. The leftmost column $j$ with $c_j +1 \in R^-$ is column 3 (which was originally column 4) with sum 6. We add ones in $(7,3)$ and $(8,3)$. As $r_5 = r_4$, we take $I = \{6, 4\}$, so we delete ones from $(4,3)$ and $(6,3)$. We then delete column 3.

\textbf{Step 5.} We have $R^+= \{5, 6\}$ and $R^- =\{11, 12\}$. As $|R^+| \leq |R^-|$, we execute an A-step. The rightmost column $j$ with $c_j \in R^+$ is column 7 (which was originally column 9) with sum 6. We deletes ones from $(5,7)$ and $(6,7)$, and we add ones in $(11,7)$ and $(12,7)$. We then delete column 7.

Now all $d_i$ have become 0, so we are done. We put back the deleted columns in their original places and find the set $F_2$ that satisfies the original line sums, see Figure \ref{figconstructieeind}.

\section{Boundary length of the constructed solution}\label{boundary}

In this section we prove upper bounds on the length of the boundary of the set that results from the construction described in the previous section.

\begin{theorem}\label{thmgrensrandalpha}
Let row sums $\mathcal{R} = (r_1, r_2, \ldots, r_m)$ and column sums $\mathcal{C} = (c_1, c_2, \ldots, c_n)$ be given, where $n=r_1 \geq r_2 \geq \ldots \geq r_m$ and $m=c_1 \geq c_2 \geq \ldots \geq c_n$. Assume that the line sums are consistent. Let $\alpha = \alpha(\mathcal{R}, \mathcal{C})$. For the set $F_2$ constructed in Theorem \ref{thmconstruction} we have
\[
L_h(F_2) \leq 2n+2\alpha, \qquad L_v(F_2) \leq 2m + 2\alpha.
\]
\end{theorem}

\begin{proof}
Let $F_1$ be the uniquely determined neighbour corresponding to the line sums $(\mathcal{R},\mathcal{C})$. Starting with $F_1$, we apply $\varphi$ repeatedly, moving ones in several columns accordingly and deleting those columns. After that, to analyse what happens to the boundary, we start again with $F_1$ and repeat the entire procedure, moving exactly the same ones, but this time keeping the columns that were supposed to be deleted.

The length of the horizontal boundary of $F_1$ is equal to $2n$, since there are $n$ columns that each contain one connected set of ones. The length of the vertical boundary of $F_1$ is $2m$. Note that the ones that are moved when applying $\phi$ are always deleted from a row $i$ with $d_i > 0$ and added to a row $i$ with $d_i < 0$. In fact for each row $i$ with $d_i > 0$ ones are deleted exactly $d_i$ times during the construction, and for each row $i$ with $d_i < 0$ ones are added exactly $-d_i$ times. Therefore the total number of ones that are moved is equal to $\alpha$. We now want to show that when in one application of $\varphi$ exactly $s$ ones are moved, both the horizontal and vertical boundary do not increase with more than $2s$. From this the theorem follows.

We will only consider what happens at an A-step; the other case is analogous. So suppose we execute an A-step and move $s$ ones, while either the horizontal or vertical boundary increases by more than $2s$. First consider the horizontal boundary. Since the ones in the rows $i$ with $i_1 \leq i \leq c_j$ are removed, and there never was a one in $(c_j+1, j)$, this does not yield any additional boundary. Adding the ones in the rows $i$ with $i \in I$ may yield additional boundary, but only 2 for each one that is added, so at most $2s$ in total.

So we may assume that the vertical boundary has increased by more than $2s$. Adding the ones leads to additional vertical boundary of at most $2s$, so deleting the ones must also have led to additional boundary. This means that there was a one in $(i,j)$, which is now deleted, while there are still ones in $(i,j-1)$ and $(i,j+1)$. As $d_i >0$, those ones cannot have been added during earlier steps in the construction, so they must have been there from the beginning. This means in particular that $c_{j+1} \geq i \geq i_1$, while also $c_{j+1} \leq c_j \leq i_2$, so $c_{j+1} \in R^+$. But $j$ was chosen maximally such that $c_j \in R^+$, so apparently column $j+1$ was in the original construction deleted in an earlier application of $\varphi$.

Suppose this earlier application has been an A-step. Since rows $l$ with $d_l \leq 0$ at some point in the construction can never have $d_l > 0$ at a later point in the construction, we know that all rows $l$ with $i_1 \leq l \leq c_{j+1}$ were contained in $R^+$ in this earlier application of $\varphi$. In particular should the one in $(i,j+1)$ have been moved during this step. So this is impossible.

Now suppose that the earlier application has been a B-step. Then column $j+1$ can only have been chosen to execute this step in if $d_{c_{j+1}+1} < 0$. Since $c_{j+1} \leq c_j$ and $d_{c_j} >0$ (now, and therefore also earlier), we then must have $c_j = c_{j+1}$. Hence $d_{c_j+1} < 0$, which means that to execute this B-step column $j$, rather than column $j+1$, should have been chosen. So this case is impossible as well.

We conclude that the vertical boundary has increased by at most $2s$ as well, and this completes the proof of the theorem.
\end{proof}

In light of this theorem it is interesting to note that $\alpha$ cannot become arbitrary large while $n$ and $m$ are fixed. In fact, we have the following result.

\begin{proposition}\label{propgrensalpha}
Let row sums $\mathcal{R} = (r_1, r_2, \ldots, r_m)$ and column sums $\mathcal{C} = (c_1, c_2, \ldots, c_n)$ be given, where $n=r_1 \geq r_2 \geq \ldots \geq r_m$ and $m=c_1 \geq c_2 \geq \ldots \geq c_n$. Assume that the line sums are consistent. Let $\alpha = \alpha(\mathcal{R}, \mathcal{C})$. Then
\[
\alpha \leq \frac{(m-1)(n-1)}{4}.
\]
\end{proposition}

\begin{proof}
For $i = 1, 2, \ldots, m$, let $b_i = \#\{j: c_j \geq i\}$ and $d_i = b_i-r_i$. Let $a$ be the number of rows (indices $i$) with $d_i > 0$ and $b$ the number of rows with $d_i<0$. We assume $\alpha>0$, so $a, b >0$. Define $d^+ = \max\{d_i: d_i>0\}$ and $d^- = \max\{-d_i : d_i<0\}$. We have $b_1 = n = r_1$, so $d_1 = 0$, hence $a+b \leq m-1$.

Now we prove that $d^+ + d^- \leq n-1$. Let $k$ and $l$ be such that $b_k-r_k = d^+$ and $r_l-b_l=d^-$. First suppose $k<l$. Then since $r_1 \geq r_2 \geq \ldots \geq r_m$ and $b_1 \geq b_2 \geq \ldots \geq b_m$ we have $b_1 \geq b_k = b_k-r_k+r_k = d^+ +r_k$ and $-b_m \geq -b_l = r_l -b_l-r_l = d^- -r_l$, hence
\[
d^+ + d^- \leq (b_1-r_k) + (-b_m+r_l) \leq b_1-b_m \leq n-1.
\]
If on the other hand $k>l$, then $r_1 \geq r_l = r_l-b_l+b_l = d^- + b_l$ and $-r_m \geq -r_k = b_k -r_k-b_k = d^+ -b_k$, and hence
\[
d^+ + d^- \leq (-r_m+b_k) + (r_1-b_l) \leq r_1-r_m \leq n-1.
\]

Now note that we have
\[
\alpha = \sum_{d_i>0} d_i = \sum_{d_i<0} (-d_i),
\]
so
\[
\alpha^2 = \left( \sum_{d_i>0} d_i \right) \left( \sum_{d_i<0} (-d_i) \right) \leq \big( a \cdot d^+ \big) \big( b \cdot d^- \big) = \big(a \cdot b \big) \big(d^+ \cdot d^- \big)
\]
\[
\leq \left( \frac{a+b}{2} \right)^2 \left( \frac{d^++d^-}{2} \right)^2 \leq \left( \frac{m-1}{2} \right)^2 \left( \frac{n-1}{2} \right)^2.
\]
Therefore
\[
\alpha \leq \frac{(m-1)(n-1)}{4}.
\]
\end{proof}

In case of large $\alpha$, the construction of Theorem \ref{thmconstruction} actually gives a much smaller horizontal boundary than the bound in Theorem \ref{thmgrensrandalpha}, as the following theorem shows.

\begin{theorem}\label{thmgrenshorizontalerand}
Let row sums $\mathcal{R} = (r_1, r_2, \ldots, r_m)$ and column sums $\mathcal{C} = (c_1, c_2, \ldots, c_n)$ be given, where $n \geq 2$, $n=r_1 \geq r_2 \geq \ldots \geq r_m$ and $m=c_1 \geq c_2 \geq \ldots \geq c_n$. Assume that the line sums are consistent. For the set $F_2$ constructed in Theorem \ref{thmconstruction} we have
\[
L_h(F_2) \leq 4n-4.
\]
\end{theorem}

\begin{proof}
We will prove this by induction on $n$. Let $\alpha = \alpha(\mathcal{R}, \mathcal{C})$. If $\alpha > 0$, then there are $l_1$ and $l_2$ such that $2 \leq l_1 < l_2$ and $d_{l_1}>0$ and $d_{l_2} < 0$. Then
\[
b_1 \geq b_{l_1} \geq r_{l_1}+1 \geq r_{l_2}+1 \geq b_{l_2} +2 \geq 1 + 2 = 3.
\]
Hence $n \geq 3$. So when $n=2$, we have $\alpha =0$ and the construction gives $F_2 = F_1$, with $L_h = 2n = 4n-2n = 4n-4$.

Now let $k \geq 3$ and suppose that we have proved the theorem in case $n < k$. Let $n=k$. Let $F_1$ be the uniquely determined neighbour corresponding to the line sums $(\mathcal{R},\mathcal{C})$. We apply $\varphi$ to $F_1$ once. Assume without loss of generality that an A-step is executed in column $j$.

First suppose that $I$ consists of consecutive numbers. Then after moving the ones in column $j$, the length of the horizontal boundary in this column is equal to 4. When we delete this column, we are left with $k-1$ columns, so we can apply the induction hypothesis, which yields that the total length of the horizontal boundary at the end of the construction will be
\[
L_h \leq 4(k-1) - 4 + 4 = 4k-4.
\]

Now suppose that $I$ does not consist of consecutive numbers. Then we know that $I$ is of the form $I = \{i_3, i_3+1, \ldots, t_1-1\} \cup \{t_3, t_3+1, \ldots, t_2\}$. So after moving the ones, the length of the horizontal boundary in column $j$ is equal to 6. Also, we know in particular that the one in $(c_j, j)$ was deleted and a one was added in $(i_3, j)$.

The new parameters, after moving the ones and deleting column $j$, we denote by $r_i'$, $b_i'$ and $d_i'$. The construction will in later steps execute an A-step in at most $d_{i_3-1}'$ columns with sum $i_3-1$ and a B-step in at most $-d_{i_3}'$ columns with sum $i_3-1$. On the other hand, we currently have $b_{i_3-1}' - b_{i_3}'$ columns with sum $i_3-1$.

We know that $r_{i_3-1} \geq r_{i_3}$, and $r_{i_3}' = r_{i_3}-1$. Both in the case $c_j = i_3-1$ and in the case $c_j < i_3-1$, we have $r_{i_3-1}' = r_{i_3-1}$, so
\[
(b_{i_3-1}' - b_{i_3}') - (d_{i_3-1}'-d_{i_3}') = r_{i_3-1}' - r_{i_3}' = r_{i_3-1} - r_{i_3} + 1 \geq 1.
\]
This means that there is at least one column with sum $i_3-1$ in which none of the later steps of the construction will be executed. This column will at the end of the construction therefore still have a horizontal boundary of length 2. If we delete this column entirely and then do the construction, exactly the same steps will be carried out. After all, the deleted column would never have been chosen to execute a step in anyway; also, deleting the column does not influence the choice of the set $I$ in each step of the construction, as the only difference between the row sums of two consecutive rows that is changed, is between rows $i_3-1$ and $i_3$, but as $d_{i_3-1} \geq 0$ and $d_{i_3} < 0$, these rows will never both be in $R^+$ or both be in $R^-$.

By applying the induction hypothesis to the new situation with $n=k-2$, we find that the total length of the horizontal boundary at the end of the construction will be
\[
L_h \leq 4(k-2) - 4 + 6 + 2 = 4k-4.
\]

This completes the induction step.
\end{proof}

Unfortunately, we cannot prove a similar result for the vertical boundary. In fact, we can find examples for which our construction gives a vertical boundary with a length as large as $\frac{4}{9}m^2 + \frac{4}{9}m + \frac{10}{9}$, see Example \ref{exvoorbeeld5}. However, we believe that there always exists a solution with a small boundary length, both horizontal and vertical.

\begin{conjecture}\label{conjecture}
Let row sums $\mathcal{R} = (r_1, r_2, \ldots, r_m)$ and column sums $\mathcal{C} = (c_1, c_2, \ldots, c_n)$ be given, where $n=r_1 \geq r_2 \geq \ldots \geq r_m$ and $m=c_1 \geq c_2 \geq \ldots \geq c_n$. Assume that the line sums are consistent. There exists a set $F_3$ with line sums $(\mathcal{R},\mathcal{C})$ for which
\[
L_h(F_3) \leq 4n-4, \qquad L_v(F_3) \leq 4m-4.
\]
\end{conjecture}

\section{Examples}\label{examples}

We give two families of examples for which we can prove that the construction of Theorem \ref{thmconstruction} gives the smallest possible length of the boundary.

\begin{example}\label{exvoorbeeld1}
Let the number of columns $n$ be odd and let $m=n$. Define line sums
\[
\mathcal{C} = \mathcal{R} = (n, n-1, n-1, n-3, n-3, \ldots, 4, 4, 2, 2).
\]
We calculate
\[
(b_1, b_2, \ldots, b_n) = (n, n, n-2, n-2, \ldots, 3, 3, 1),
\]
\[
(d_1, d_2, \ldots, d_n) = (0, +1, -1, +1, -1, \ldots, +1, -1, +1, -1).
\]
So $\alpha = \alpha(\mathcal{R}, \mathcal{C}) = \frac{n-1}{2}$. Theorem \ref{thmgrensrandalpha} tells us that the set $F_2$ constructed with Theorem \ref{thmconstruction} satisfies
\[
L_h(F_2) \leq 2n+2\alpha = 3n-1, \qquad L_v(F_2) \leq 2m + 2\alpha = 3n-1.
\]
On the other hand, by \cite[Corollary 1]{birgit5} we know that for any set $F$ with these line sums, we have
\[
L_h(F) \geq 2n + \frac{n-1}{2}\cdot (1-(-1)) + 2 \cdot 0 = 3n-1,
\]
and by symmetry also $L_v(F) \geq 3n-1$. This shows that $F_2$ has the smallest boundary among all sets $F$ with these line sums. See for the constructed set $F_2$ in the case that $n=9$ Figure \ref{figvoorbeeld1}. (This example is in fact a slightly modified version of \cite[Example 3]{birgit5}.)
\end{example}

\begin{figure}
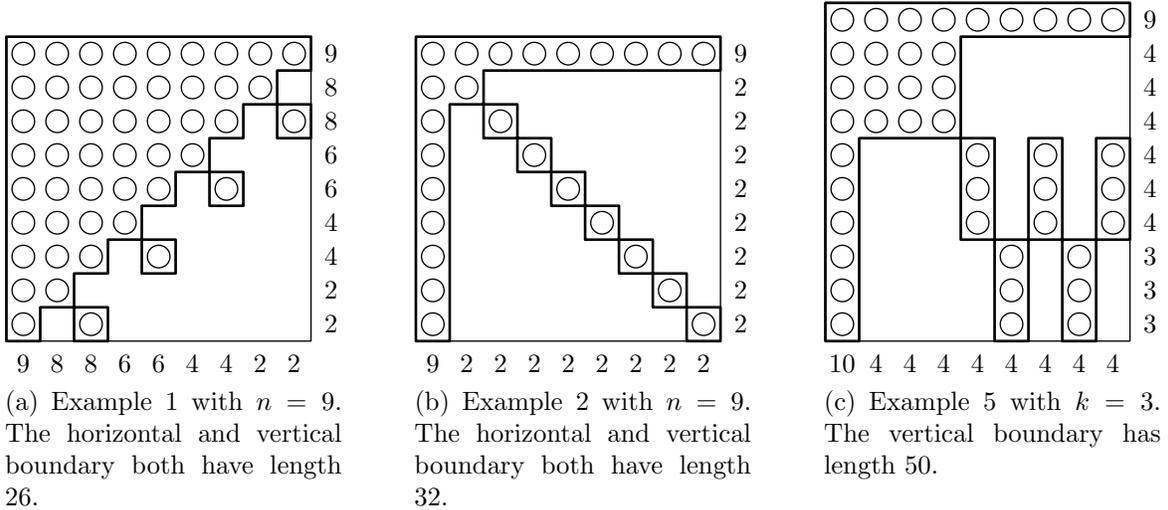

  \begin{center}
    \subfigure[Example \ref{exvoorbeeld1} with $n=9$. The horizontal and vertical boundary both have length 26.]{\label{figvoorbeeld1}\includegraphics{plaatje.1}}
    \qquad
    \subfigure[Example \ref{exvoorbeeld2} with $n=9$. The horizontal and vertical boundary both have length 32.]{\label{figvoorbeeld2}\includegraphics{plaatje.2}}
    \qquad
    \subfigure[Example \ref{exvoorbeeld5} with $k=3$. The vertical boundary has length 50.]{\label{figvoorbeeld5}\includegraphics{plaatje.3}}
  \end{center}
  \caption{The constructed sets $F_2$ for some of the examples.}
  \label{figvoorbeelden}
\end{figure}

\begin{example}\label{exvoorbeeld2}
Let $m=n \geq 2$. Define line sums
\[
\mathcal{C} = \mathcal{R} = (n, 2, 2, 2, \ldots, 2, 2).
\]
We calculate
\[
(b_1, b_2, \ldots, b_n) = (n, n, 1, 1, \ldots, 1, 1),
\]
\[
(d_1, d_2, \ldots, d_n) = (0, n-2, -1, -1, \ldots, -1).
\]
So $\alpha = \alpha(\mathcal{R}, \mathcal{C}) = n-2$. Theorem \ref{thmgrensrandalpha} tells us that the set $F_2$ constructed with Theorem \ref{thmconstruction} satisfies
\[
L_h(F_2) \leq 2n+2\alpha = 4n-4, \qquad L_v(F_2) \leq 2m + 2\alpha = 4n-4.
\]
On the other hand, by \cite[Corollary 1]{birgit5} we know that for any set $F$ with these line sums, we have
\[
L_h(F) \geq 2n + 2(n-2) = 4n-4,
\]
and by symmetry also $L_v(F) \geq 4n-4$. This shows that $F_2$ has the smallest boundary among all sets $F$ with these line sums. See for the constructed set $F_2$ in the case that $n=9$ Figure \ref{figvoorbeeld2}. (This example is in fact a slightly modified version of \cite[Example 4]{birgit5}.)
\end{example}

We can generalise Example \ref{exvoorbeeld2} to larger $\alpha$, in which case the bound of Theorem \ref{thmgrensrandalpha} is no longer sharp. However, in this case we can use Theorem \ref{thmgrenshorizontalerand} to prove that the horizontal boundary is the smallest possible, as shown below.

\begin{example}\label{exvoorbeeld3}
Let $k$ be a positive integer and let $m=kn-k+1$. Define line sums
\[
\mathcal{C} = (kn-k+1, k+1, k+1, \ldots, k+1, k+1),
\qquad
\mathcal{R} = (n, 2, 2, \ldots, 2).
\]
We calculate
\[
(b_1, b_2, \ldots, b_m) = (\underbrace{n, n, \ldots, n}_{k+1 \text{ times}}, \underbrace{1, 1, \ldots, 1}_{kn-2k \text{ times}}),
\]
\[
(d_1, d_2, \ldots, d_m) = (0, \underbrace{n-2, n-2, \ldots, n-2}_{k \text{ times}}, \underbrace{-1, -1, \ldots, -1}_{kn-2k \text{ times}}).
\]
Theorem \ref{thmgrenshorizontalerand} tells us that the set $F_2$ constructed with Theorem \ref{thmconstruction} satisfies
\[
L_h(F_2) \leq 4n-4.
\]
On the other hand, by \cite[Corollary 1]{birgit5} we know that for any set $F$ with these line sums, we have
\[
L_h(F) \geq 2n + 2(n-2) = 4n-4.
\]
This shows that $F_2$ has the smallest horizontal boundary among all sets $F$ with these line sums.
\end{example}

The next example shows that the upper bound on $\alpha$ given in Proposition \ref{propgrensalpha} can be achieved.

\begin{example}\label{exvoorbeeld4}
Let $k$ be a positive integer and let $m=n=2k+1$. Define line sums
\[
\mathcal{C} = \mathcal{R} =(2k+1, k+1, k+1, \ldots, k+1).
\]
We calculate
\[
(b_1, b_2, \ldots, b_m) = (\underbrace{2k+1, 2k+1, \ldots, 2k+1}_{k+1 \text{ times}}, \underbrace{1, 1, \ldots, 1}_{k \text{ times}}),
\]
\[
(d_1, d_2, \ldots, d_m) = (0, \underbrace{k, k, \ldots, k}_{k \text{ times}}, \underbrace{-k, -k, \ldots, -k}_{k \text{ times}}).
\]
Hence
\[
\alpha = \alpha(\mathcal{R}, \mathcal{C}) = k^2 = \frac{(m-1)(n-1)}{4}.
\]
\end{example}

Finally we show by an example that the vertical boundary of the set $F_2$ constructed in Theorem \ref{thmconstruction} can become quite large, so it is not possible to prove a similar result as Theorem \ref{thmgrenshorizontalerand} for the vertical boundary.

\begin{example}\label{exvoorbeeld5}
Let $k$ be a positive integer and let $m=3k+1$, $n=3k$. Define line sums
\[
\mathcal{C} =(3k+1, k+1, k+1, \ldots, k+1), \qquad \mathcal{R} = (3k, \underbrace{k+1, k+1, \ldots, k+1}_{2k \text{ times}}, \underbrace{k, k, \ldots, k}_{k \text{ times}}).
\]
We calculate
\[
(b_1, b_2, \ldots, b_m) = (\underbrace{3k, 3k, \ldots, 3k}_{k+1 \text{ times}}, \underbrace{1, 1, \ldots, 1}_{2k \text{ times}}),
\]
\[
(d_1, d_2, \ldots, d_m) =\]\[ (0, \underbrace{2k-1, 2k-1, \ldots, 2k-1}_{k \text{ times}}, \underbrace{-k, -k, \ldots, -k}_{k \text{ times}}, \underbrace{-(k-1), -(k-1), \ldots, -(k-1)}_{k \text{ times}}).
\]
Hence $\alpha = \alpha(\mathcal{R}, \mathcal{C}) = 2k^2-k$.

The construction executes $2k-1$ times an A-step, in each of the columns $3k$, $3k-1$, \ldots, $k+2$. In the first step (and every odd-numbered step after that) we have $I = \{k+2, k+3, \ldots, 2k+1\}$. At the beginning of the second step, however, the row sums in rows $k+2$, $k+3$, \ldots, $3k+1$ are all equal, so we have $I = \{2k+2, 2k+3, \ldots, 3k+1\}$. The same holds for every other even-numbered step. This means that at the end of the construction, the vertical boundary in each of the rows $k+2$, $k+3$, \ldots, $2k+1$ will be equal to $2(k+1)$, while the vertical boundary in each of the rows $2k+2$, $2k+3$, \ldots, $3k+1$ will be equal to $2k$. Adding the boundary of 2 in each of the rows 1, 2, \ldots, $k+1$, we find
\[
L_v(F_2) = (k+1) \cdot 2 + k \cdot 2(k+1) + k \cdot 2k = 4k^2 + 4k + 2.
\]
This is not linear in $m = 3k+1$. It is in fact equal to $\frac{4}{9}m^2 + \frac{4}{9}m + \frac{10}{9}$. See for the constructed set $F_2$ in the case that $k=3$ Figure \ref{figvoorbeeld5}.

It is clear that in fact there exists a set $F$ with the same line sums, but with a much smaller vertical boundary, which supports Conjecture \ref{conjecture}.
\end{example}

\section{Generalising the results for arbitrary $c_1$ and $r_1$}\label{generalisation}

In all results of the previous sections, we used the condition that $c_1 = m$ and $r_1 = n$. This is purely for convenience; it is not a necessary condition. We can easily generalise the results for the case that these conditions do not necessarily hold.

Consider a set $F$ with row sums $\mathcal{R} = (r_1, r_2, \ldots, r_m)$ and column sums $\mathcal{C} = (c_1, c_2, \ldots, c_n)$, where $r_1 \geq r_2 \geq \ldots \geq r_m$ and $c_1 \geq c_2 \geq \ldots \geq c_n$, but not necessarily $c_1 = m$ and $r_1=n$. Let $F'$ be a set that is equal to $F$, except that we add a full row with index 0 and a full column with index 0, i.e.
\[
F' = F \cup \{(0,j) : 0 \leq j \leq n\} \cup \{(i,0) : 1 \leq i \leq m\}.
\]
The row sums of $F'$ are
\[
\mathcal{R}' = (r_0', r_1', r_2', \ldots, r_m') = (n, r_1+1, r_2+1, \ldots, r_m+1).
\]
and the column sums of $F'$ are
\[
\mathcal{C}' = (c_0', c_1', c_2', \ldots, c_n') = (m, c_1+1, c_2+1, \ldots, c_n+1).
\]
It is easy to see that $\alpha(\mathcal{R}',\mathcal{C}') = \alpha(\mathcal{R},\mathcal{C})$. Now consider the length of the horizontal boundary. For every $j$ with $(1,j) \in F$, the horizontal boundary in column $j$ of $F'$ is equal to the horizontal boundary of column $j$ in $F$. For every $j$ with $(1,j) \not\in F$, however, the horizontal boundary in column $j$ of $F'$ is 2 larger than the horizontal boundary in column $j$ of $F$. (This also holds for column 0, where the horizontal boundary of $F$ had length 0 and the horizontal boundary of $F'$ has length 2.) Hence
\[
L_h (F') = L_h (F) + 2(n+1-r_1).
\]
Analogously, we have
\[
L_v(F') = L_v(F) = 2(m+1-c_1).
\]
By applying Theorems \ref{thmgrensrandalpha} and \ref{thmgrenshorizontalerand} as well as Proposition \ref{propgrensalpha} to $F'$ (with $n+1$ columns and $m+1$ rows), we acquire the following results.

\begin{proposition}
Let row sums $\mathcal{R} = (r_1, r_2, \ldots, r_m)$ and column sums $\mathcal{C} = (c_1, c_2, \ldots, c_n)$ be given, where $r_1 \geq r_2 \geq \ldots \geq r_m$ and $c_1 \geq c_2 \geq \ldots \geq c_n$. Assume that the line sums are consistent. Let $\alpha = \alpha(\mathcal{R}, \mathcal{C})$. Then
\[
\alpha \leq \frac{mn}{4}.
\]
\end{proposition}

\begin{theorem}
Let row sums $\mathcal{R} = (r_1, r_2, \ldots, r_m)$ and column sums $\mathcal{C} = (c_1, c_2, \ldots, c_n)$ be given, where $r_1 \geq r_2 \geq \ldots \geq r_m$ and $c_1 \geq c_2 \geq \ldots \geq c_n$. Assume that the line sums are consistent. Let $\alpha = \alpha(\mathcal{R}, \mathcal{C})$. Then there exists a set $F_2$ satisfying these line sums such that
\[
L_h(F_2) \leq \min( \ 2r_1+2\alpha, \ 2r_1+2n-2 \ )
\]
and
\[
L_v(F_2) \leq 2c_1 + 2\alpha.
\]
\end{theorem}

\end{document}